\documentclass{article}

\usepackage[utf8]{inputenc}

\usepackage{amsthm}
\usepackage{amsmath}

\usepackage{amssymb}
\usepackage{amsfonts}
\newtheorem{tw}{Theorem}
\usepackage{indentfirst}
\usepackage{mathtools}

\def\aut{\operatorname{Aut}}
\def\Im{\operatorname{Im}}
\def\Re{\operatorname{Re}}
\def\Mod{\operatorname{mod}}
\newtheorem{lem}[tw]{Lemma}
\newtheorem{cor}[tw]{Corollary}
\newtheorem*{df*}{Definition}

\newtheorem*{re*}{Remark}
\title{CHARACTERISATION OF GEODESIC-PRESERVING FUNCTIONS}
\author{MARCIN TOMBIŃSKI}
\date{}

\begin{document}

\maketitle
\begin{abstract}
 Let $\Omega_1$, $\Omega_2$ be two domains in $\mathbb{C}^n$ with Kobayashi metrics $k_{\Omega_i}$ and consider a holomorphic mapping $f \in \mathcal{O}(\Omega_1,\Omega_2)$. Let $\mathfrak{F}_1$ and $\mathfrak{F}_2$ be families of geodesics defined on $\Omega_1$ and $\Omega_2$ respectively, where a geodesic between $z$ and $w$ in $\Omega_i$ is the length minimizing curve for the metric $k_{\Omega_i}$. We say that a holomorphic mapping \textit{preserves geodesics} if for any geodesic $\gamma_1$ in $\mathfrak{F}_1$ its image is a subset of a geodesic $\gamma_2$ in $\mathfrak{F}_2$ ($f(\gamma_1)\subset \gamma_2$). 
 
 We aim to characterise the family of such mappings when $\mathfrak{F}_1$ and $\mathfrak{F}_2$ are the families of Kobayashi geodesics passing through a point in the unit disc $\mathbb{D}$ or in the unit ball $\mathbb{B}^n$. Some
additional results are given in the complex plane  $\mathbb{C}$ and $\mathbb{C}^n$.

\end{abstract}
\section{Introduction}
Let $\Omega_1$, $\Omega_2$ be two domains in $\mathbb{C}^n$ and consider a holomorphic mapping  $f \in \mathcal{O}(\Omega_1,\Omega_2)$. A geodesic $\gamma$ between two points $z$ and $w$ in $\Omega_i$ is the length minimizing curve joining these two points.

The authors of \cite{BZK} studied the following problem. Assume that $M$ and $N$ are two Kobayashi complete manifolds. Let $F: M\rightarrow N$ be a holomorphic mapping that isometrically maps every complex geodesic from a collection $\mathfrak{F}$ of complex geodesics in $M$ onto a complex geodesic in $N$. For which $M$, $N$ and $\mathfrak{F}$ are all holomorphic maps $F$ biholomorphisms? This property is called  the \textit{slice rigidity property}. Complex geodesics are not to be confused with geodesics, as they are holomorphic maps from the unit disc. In a later work \cite{BKZ}, the authors give a precise characterisation of the condition when Kobayashi isometries between families of geodesics are biholomorphisms. 

The problems addressed in both studies revolve around isometries restricted to families of real or complex geodesics. The isometry condition imposed on a map is restrictive. In \cite{BZK} the mappings are limited to Kobayashi isometries. Motivated by this, the author studied the same problem under a weaker condition: we ask which holomorphic mappings preserve geodesics \textit{geometrically}, in the sense that they send each geodesic in one family into a geodesic in another. 
 
 \begin{df*}
     
  Given families of geodesics $\mathfrak{F}_1$ and $\mathfrak{F}_2$ on $\Omega_1$ and $\Omega_2$, we say that a map $f$ is \textit{geodesic-preserving} between two families of geodesics $\mathfrak{F}_1$ and $\mathfrak{F}_2$ if for every geodesic (the image of a geodesic linking two points) $\gamma_1 \in \mathfrak{F_1}$ there exists a geodesic $\gamma_2 \in \mathfrak{F_2}$ such that $f(\gamma_1)\subset\gamma_2$. We denote the family of all such mappings by $GP(\mathfrak{F}_1,\mathfrak{F}_2)$. 
  \end{df*}
  This set is never empty: every constant map sending $\Omega_1$ to a point lying on some geodesic in $\Omega_2$ trivially satisfies the condition.

 In this paper we focus on the most fundamental Kobayashi-hyperbolic domains: the unit disc $\mathbb{D}$ and the unit ball $\mathbb{B}^n$; and holomorphic selfmaps of these domains. These domains are especially significant because the arguments in \cite{BKZ, BZK} reduce the general problem to these model cases via scaling methods. 
 
A geodesic is a map from an interval $[a;b]$ in $\mathbb{R}$ to $\Omega$. When the interval is bounded, we  call $\gamma$ a \emph{geodesic segment}. Since the studied domains $\Omega$ are Kobayashi-complete domains, for any two points in $\Omega$ there exists a geodesic segment joining these two points. Additionally we can take the extension maps from the whole of $\mathbb{R}$ to $\Omega$, with the limits lying on the boundary $\partial\Omega_i$ at $\pm \infty$. A geodesic whose interval is the whole real axis is called a geodesic line. Identify a geodesic with its image $\gamma$.
 
 Therefore, for any point $a$ in its closure $\overline{\Omega}$, we define the corresponding family of geodesics denoted by $$\mathfrak{G}^\Omega_a :=\{ \gamma: a\in \overline{\gamma}, \text{where } \gamma \text{ is the image of a geodesic line in } \Omega \}.$$
 The point $a$ will be referred to as the intersection point of the family $\mathfrak{G}^\Omega_a$.

 Using geometric properties of geodesics through the origin and the invariance of geodesics under automorphisms, we eventually obtain a characterisation of all geodesic-preserving functions between families of geodesics passing through arbitrary points in the ball. For a point $p \in \mathbb{B}^n$ we denote by $M_p$ any automorphism of the unit ball moving 0 to $p$.

 \begin{tw}
Geodesic-preserving mappings $f \in GP(\mathfrak{G}^{\mathbb{B}^n}_a,\mathfrak{G}^{\mathbb{B}^n}_b)$, where $a,b \in \mathbb{B}^n$, are of the following form 
     $$f(z)=M_b\circ g \circ M_a^{-1} (z),$$
     where $g: \mathbb{B}^n \rightarrow \mathbb{B}^n$ is a homogeneous polynomial mapping of degree $m$.
     
 \end{tw}
 
In one complex variable, we also characterise geodesic-preserving functions for families of geodesics with intersection points on the boundary of the disc. In the cases where the intersection points of families of geodesics lie on the boundary of the disc, we will refer to M\"obius transformations of the closed unit disc fixing a point on the boundary $a \in \partial\mathbb{D} $ denoted by $T_a$. These are fractional linear self-maps of the unit disc $T_a$ such that $T_a(a)=a$ for a given $a \in \partial\mathbb{D}$.  
\begin{re*}   
 M\"obius transformations are not to be confused with a smaller class, namely, M\"obius maps which are automorphisms of the unit disc.
 \end{re*}
 \begin{tw}
     Geodesic-preserving functions $f \in GP ( \mathfrak{G}^\mathbb{D}_a, \mathfrak{G}^\mathbb{D}_b)$, where $a,b \in \partial\mathbb{D}$, are of the form $f(z)=\sigma \circ T_a$, where $\sigma$ is a rotation mapping $a$ to $b$.
 \end{tw}
 In several variables, the results rely on the geometric aspect of the families of geodesics in the Siegel domain obtained by the Cayley transform centered at $e_1 = (1,0,\dots, 0)$.
 \begin{tw}
    Let $a,b \in \partial\mathbb{B}^n$. The geodesic-preserving functions $f \in GP ( \mathfrak{G}^{\mathbb{B}^n}_a, \mathfrak{G}^{\mathbb{B}^n}_b)$ are of the form 

$$f(z)=\Phi_{b} \circ C^{-1}_{e_1} \circ  F \circ 
C_{e_1} \circ \Phi_a^{-1} (z),$$
where $C_{e_1}$ is the Cayley transform around $e_1$, $F=(F_1, \tilde{F})$ satisfies $F_1 (z_1,z') = mz_1+b(z')$, $(z_1,z') \in \mathbb{C} \times \mathbb{C}^{n-1}$, $\tilde{F} \in \mathcal{O}( \mathbb{C}^{n-1} )$ such that
$$|\tilde{F}(z')|^2 \leq m |z'|^2 +\Im(b(z')) ;$$
and $\Phi_\zeta \in \aut(\overline{\mathbb{B}^n})$, mapping $e_1$ to $\zeta$.

\end{tw}

Additional results are presented for open simply connected domains in $\mathbb{C}$ and strictly linearly convex in $\mathbb{C}^n$.

 \section{Geodesic-preserving functions in the unit disc}
For the unit disc $\mathbb{D}$ and the unit ball $\mathbb{B}^n$, the families of geodesics give foliations of these sets.

In each case, we begin by considering the case of geodesic-preserving functions between families of geodesics passing through the origin, i.e. $\mathfrak{G}^\mathbb{D}_0$ and $\mathfrak{G}^\mathbb{B}_0$. 
The geodesics passing through $0$ in the unit disc are exactly its diameters. Thus, one expects that geodesic-preserving functions will rotate and possibly contract geodesics. This expectation is confirmed by the next theorem.
\begin{tw} \label{1.1}
    Geodesic-preserving functions $f \in GP(\mathfrak{G}^\mathbb{D}_0)$ are homogeneous polynomials, that is
    \begin{equation}\label{1}
        f(z) = \alpha z^n,
    \end{equation}
    where $\alpha \in \bar{\mathbb{D}}$ and $n \in \mathbb{N}$.
\end{tw}
\begin{proof}
    Let $f \in GP(\mathfrak{G}^\mathbb{D}_0)$. 
    
    Assume that $f$ is not constant. The holomorphicity of $f$ ensures that $f$ has a Taylor expansion $f (z) = \sum_{n=0}^{\infty} a_{n}z^n$, $a_n \in \mathbb{C}$, $n \in \mathbb{N}$. We may assume that the real axis is mapped into itself (we may compose $f$ with a rotation). Then  all the coefficients $a_n$ are real. 
    
    Let $k$ be the smallest index  such that $a_k \neq 0$. Take $z = re^{i\phi}$, $\phi \neq 0 [\Mod   \pi]$, belonging to a geodesic (a diameter), $r \in (-1,1)$. Then $$f(z)= a_kr^ke^{ik\phi} +a_{k+1}r^{k+1}e^{i(k+1)\phi} + \dots \ .$$ The image of a diameter is contained in another diameter, therefore $f(re^{i\phi})$ must lie on a line for varying $r$. Since the image is contained in a fixed diameter, the argument of $e^{-ik\phi} f(r e^{i\phi})$ must remain constant for sufficiently small $r$. Hence all higher-order coefficients vanish.

Hence 
    $$f(z) = \alpha z^n,\  \alpha \in \overline{\mathbb{D}}, \ n \in \mathbb{N}.$$
    This completes the proof.

\end{proof}

When $n$ is even and $\alpha \in \partial \mathbb{D}$ the geodesic (diameter) becomes folded into a ray because $2k\phi \equiv 2k (\pi+\phi) [\Mod 2\pi]$, $k \in \mathbb{N}$, whereas when $n$ is odd, the map $z\mapsto z^n$ maps a diameter onto another diameter.
\begin{cor}
    Functions $f \in GP(\mathfrak{G}^\mathbb{D}_0)$ that satisfy the stronger condition that for any $\gamma_1 \in \mathfrak{G}_0^\mathbb{D}$ there exists $\gamma_2 \in \mathfrak{G}_0^\mathbb{D}$ such that $f(\gamma_1)= \gamma_2$, are of the form $f(z)= \alpha z^{2n+1}$, where $\alpha \in \partial\mathbb{D}$ and $n \in \mathbb{N}$.
\end{cor}

\begin{cor}
    If $f \in GP(\mathfrak{G}^\mathbb{D}_0)$ is an automorphism, then $f$ is a rotation.  
\end{cor}

These results extend to families of geodesics $\mathfrak{G}^\mathbb{D}_a$ and $\mathfrak{G}^\mathbb{D}_b$ with $a,b \in \mathbb{D}$. Since automorphisms of the unit disc are Kobayashi isometries, they  map geodesics to geodesics. Hence $M_a(\mathfrak{G}^\mathbb{D}_0)=\mathfrak{G}^\mathbb{D}_a$ and $M_b(\mathfrak{G}^\mathbb{D}_0)=\mathfrak{G}^\mathbb{D}_b$.

\begin{tw}\label{4}
    Let $a, b \in \mathbb{D}$. Geodesic-preserving functions in the unit disc $f \in GP(\mathfrak{G}^\mathbb{D}_a, \mathfrak{G}^\mathbb{D}_b)$ are of the form $f = M_b \circ g \circ M^{-1}_a$, where $g \in GP(\mathfrak{G}^\mathbb{D}_0)$.
\end{tw}
\begin{proof}
 Observe that $f \in GP(\mathfrak{G}^{\mathbb{D}}_a,\mathfrak{G}^{\mathbb{D}}_b)$ if and only if $M_b^{-1} \circ f \circ M_a \in GP(\mathfrak{G}_0^{\mathbb{D}})$.

\end{proof}

 The geometry of the family of geodesics $\mathfrak{G}^\mathbb{D}_a$ changes when the intersection point moves to the boundary. In particular, geodesics are no longer diameters but curves asymptotic to the same boundary point. Consequently, the argument used in the interior case no longer applies.
 
 Observe that any nonconstant holomorphic geodesic-preserving function $f\in GP ( \mathfrak{G}^{\mathbb{D}}_a, \mathfrak{G}^\mathbb{D}_b)$, maps the intersection point $a$ to the intersection point $b$ and $\lim f(x)=b $ as $x$ tends to $a$.  We now classify all geodesic-preserving functions for families $\mathfrak{G}^{\mathbb{D}}_a$ with boundary intersection point $a \in \partial \mathbb{D}$.

\begin{tw}
    Let $a \in \partial\mathbb{D}$. The geodesic-preserving functions for the family of geodesics $\mathfrak{G}^\mathbb{D}_a$ are  M\"obius transformations which fix the point $a$. 
\end{tw}

\begin{proof} Let $a$ be a point on the unit circle $\partial\mathbb{D}$ and $f$ be a geodesic-preserving function for the family $\mathfrak{G}_a^\mathbb{D}$, $f\in GP(\mathfrak{G}_a^\mathbb{D})$.

    Assume $f$ is nonconstant. 
    
    Without loss of generality, by composing with an automorphism of the unit disc, we may assume that $f(0)=0$ and $a=1$. For each $p$ in $\mathbb{D}$, let 
    $$ T_p (z) = \frac{1-\bar{p}}{1-p} \frac{z-p}{1- \bar{p}z}$$
    be an automorphism mapping the geodesic through $p$ and 1 to the real axis.
   
Define
$$L_p(t)= T_{f(p)} ( f (T_p^{-1}(t))), \ t \in (-1,1).$$
Then $L_p$ maps the interval $(-1,1)$ into itself and preserves 0. Its Taylor series at 0 has only real coefficients:
$$L_p(t) = \alpha_1 t + \alpha_2 t^2 + \dots, \ \alpha_i \in \mathbb{R}.$$ Computing the derivative of $L_p$ at 0 shows that
 $$(L_p)' (0) = \frac{1-|p|^2}{1-|f(p)|^2}\frac{1-\overline{f(p)}}{1-f(p)} f'(p) \frac{1-p}{1-\bar{p}} \in \mathbb{R}$$
 for each $p$. This leads to the differential identity

 $$\frac{(1-p)^2}{(1-f(p))^2} f'(p) = C_p \in \mathbb{R}.$$
 Although $C_p$ may depend on $p$, the left side defines a holomorphic real-valued function of $p$. Since every holomorphic real-valued function on a connected domain is constant, therefore $C_p$ must be constant $C_p =C$. Hence  

    \begin{align*}
    \frac{f'(p)}{(1-f(p))^2} &= \frac{C}{(1-p)^2} \\
     \left( \frac{1}{1-f(p)} \right)' &= C \left( \frac{1}{1-p} \right)' \\
    \frac{1}{1-f(p)}&=\frac{C}{1-p}+B.
    \end{align*}
Using $f(0)=0$, the ODE yields
$$f(z)=\frac{Cz}{1-(1-C)z},$$
where $C$ has to be in the interval $(0,1]$ to ensure $f\in \mathcal{O}(\mathbb{D})$. These are the geodesic-preserving functions in $GP(\mathfrak{G}_1^{\mathbb{D}})$ that also fix 0. To obtain the full characterisation we have to compose these functions with automorphisms of the unit disc fixing 1.

\end{proof}
Next, we study mixed cases: when one family of geodesics has intersection point in the interior and the other on the boundary. If the intersection point of geodesics in the source lies in the interior of the disc while the intersection point in the target lies on the boundary, then any geodesic-preserving function must be constant because the image of interior geodesics cannot all accumulate to a boundary point unless the map collapses. 

The remaining nontrivial case is therefore when the source family intersects at the boundary and the target family intersects in the interior. This case leads to exponential-type functions after applying a Cayley transform. 

Let $f$ be a nonconstant holomorphic geodesic-preserving function between $\mathfrak{G}^\mathbb{D}_1$ and $\mathfrak{G}^\mathbb{D}_0$. Apply the Cayley transform $C(z)= \frac{1+z}{1-z}$ which maps $\mathbb{D}$ biholomorphically onto the right half-plane $\mathbb{H}=\{z: \Re(z)>0\}$. Under this map the geodesics from $\mathfrak{G}^\mathbb{D}_1$ are mapped to horizontal half-lines ($\{z_0+t:t>0,\Re(z_0)=0\}$). 

Let 
$$F=f\circ C^{-1}: \mathbb{H} \rightarrow \mathbb{D}.$$
Because geodesics are mapped to geodesics, for each fixed $z\in \mathbb{H}$ and $t\geq0$
$$\frac{F(z+t)}{F(z)} = c(t)\in \mathbb{R}$$
for any $z \in \mathbb{H}$ and $c$ is a real-valued function for $t\in \mathbb{R}^+$. Note that the function $F$ does not vanish on $\mathbb{H}$. The function $c(t)$ does not depend on $z$ because it is a holomorphic real-valued function on $z$, therefore it must be constant. This can be rewritten as $F(z+t)=F(z)c(t).$ Taking the derivative with respect to $z$ while fixing $t$ we obtain $F'(z+t)=F'(z)c(t)$. Observe that $F(z) \neq 0$ (otherwise $F \equiv0$). Dividing this equation by the previous one we get a logarithmic derivative
$$G(z)\coloneqq \frac{F'(z)}{F(z)} $$
satisfying $G(z+t)=G(z)$ for all $t>0$. As a holomorphic function invariant under real translations, $G$ must be constant. 
$$G(z)\equiv\alpha.$$
Therefore $F$ must have an exponential form $$F(z)=e^{\alpha z +\beta}.$$  
Horizontal half-lines have to be mapped to horizontal half-lines, therefore $\alpha$ must be real.
Additionally, the condition $|F(z)|<1$ becomes
$$\Re(\alpha z+\beta) <0, $$
for all $z \in \mathbb{H}$. Thus $\alpha$ is negative and $-\beta \in \overline{\mathbb{H}}$. Returning to the disc via $f=F \circ C$, we obtain the full classification.
\begin{tw}
    
The nonconstant geodesic-preserving functions between $\mathfrak{G}^\mathbb{D}_1$ and $\mathfrak{G}^\mathbb{D}_0$ are exactly
$$f(z) = e^{a (\frac{1+z}{1-z})-\beta},$$
where $a <0$ and $\Re(\beta)\geq 0$.
\end{tw}

Since geodesics in simply connected planar domains are pull-backs of geodesics in the disc, the previous classification immediately transfers via a Riemann map.

Let $\Omega\subsetneq \mathbb{C}$ be simply connected and let $F:\mathbb{D} \rightarrow \Omega$ be a Riemann map. The Kobayashi metric on $\Omega$ is defined by
$$k_\Omega(z,w)=k_\mathbb{D}(F^{-1}(z),F^{-1}(w)).$$
Geodesics in $\Omega$ are the images under $F$ of geodesics in the unit disc $\mathbb{D}$.
\begin{tw}
    Let $\Omega \subsetneq \mathbb{C}$ be an open  and simply connected set and $F:\mathbb{D} \rightarrow \Omega$ be a biholomorphism. For the family of geodesics $\mathfrak{G^\Omega_a}$, $a\in \Omega$, the geodesic-preserving functions $f$ are precisely 
    $$f= F \circ g \circ F^{-1},$$
    where $g \in GP(\mathfrak{G}^\mathbb{D}_{F^{-1}(a)})$.
\end{tw}
To extend this result to boundary points, the Riemann map must extend continuously as a one-to-one map to the boundary. This holds if and only if $\Omega$ is a Jordan domain.  
\begin{tw}
    Let $\Omega$ be a Jordan domain and let $F:\mathbb{D} \rightarrow \Omega$ be a biholomorphism. For the family of geodesics $\mathfrak{G^\Omega_a}$, $a\in \partial\Omega$, geodesic-preserving functions $f$ of $\mathfrak{G}^\Omega_a$ are exactly 
    $$f= F\circ g \circ F^{-1},$$
    where $g \in GP(\mathfrak{G}^\mathbb{D}_{F^{-1}(a)})$.
\end{tw}

\section{Geodesic-preserving mappings in the unit ball}
We now extend the study to several complex variables.
As in the one-dimensional case, we first consider the family of geodesics passing through the origin, denoted $\mathfrak{G}^{\mathbb{B}^n}_0$. A geodesic passing through the origin is a “diameter,” i.e., a segment $\gamma_\xi = \{ t \xi, t  \in (-1,1) \}$, where $\xi$ is a point in the unit sphere $\partial\mathbb{B}^n$. 
\begin{lem}
    The nonconstant holomorphic maps mapping real lines in $\mathbb{C}^n$ passing through the origin into real lines passing through the origin are homogeneous polynomial maps of degree $m$ . 
\end{lem}

\begin{proof} Let $f = (f_1, \dots, f_n)$ be holomorphic and assume that it maps every real line passing through 0 into a real line passing through 0. 
 Fix a line $L =\mathbb{C}\eta$, $\eta \in \partial\mathbb{B}^n$. The intersection $L \cap \mathbb{B}^n$ is a copy of the unit disc obtained by the mapping $\lambda \mapsto\lambda\eta$, for $\lambda \in \mathbb{D}$. Restricting $f$ to $L \cap \mathbb{B}^n \simeq \mathbb{D}$, we may consider $f$ to be a function of the disc $f:\mathbb{D} \to \mathbb{D}$. Then we may assume that $f$ maps diameters into diameters. Hence, by Theorem \ref{1.1}, $$f(\lambda\eta)=\lambda^{m(\eta)}f(\eta), \ \ \lambda\in \mathbb{D},$$ for some integer $m(\eta)$. Write the Taylor expansion of $f$ around 0
$$f(z) = \sum_{k \in \mathbb{N}} P_k(z),$$
where $P_k$ are homogeneous polynomial maps of degree $k$. Then $$f(\lambda\eta)=\sum_{k} \lambda^kP_k (\eta).$$ Since the left-hand side consists of a single monomial in $\lambda$, for every $\eta \in \partial \mathbb{B}^n$ there exists an integer $m(\eta)$ such that $P_k \equiv0 , \ k\neq m(\eta).$

Let $m$ be the minimal index such that $P_m \not\equiv 0$. Suppose there exists 
$l \neq m$ such that $P_l \not \equiv0$. Since the common zero set of two nonzero polynomial maps cannot contain the whole sphere $\partial\mathbb{B}^n$, there exists $\eta \in \partial\mathbb{B}^n$ such that $$P_m(\eta) \not\equiv0 \quad ,\quad P_l(\eta)\not\equiv 0.$$ 
But this contradicts the fact that, for each fixed $\eta$, at most one homogeneous term $P_k(\eta)$ can be nonzero. Therefore $P_k \equiv 0$ for $k \neq m$ and hence $f =P_m$.

\end{proof}

Restricting Lemma 12 to the unit ball $\mathbb{B}^n$ yields:
\begin{tw}\label{2.1}
    The geodesic-preserving maps $f=(f_1,\dots,f_n)$ for the family of geodesics $\mathfrak{G}^{\mathbb{B}^n}_0$  are constant or homogeneous polynomial maps of degree $m$.
\end{tw}
 The formulas for automorphisms of the unit ball $\mathbb{B}^n$ are given in \cite{Rud} which yields the following corollary.
\begin{cor}
    Automorphisms of $\mathbb{B}^n$ which preserve $\mathfrak{G}^{\mathbb{B}^n}_0$ are unitary mappings$$f(z) = U z,$$ where $U$ is a unitary matrix.
\end{cor}
These results can be extended to the family of geodesics with intersection points $a,b \in \mathbb{B}^n $ by composing with automorphisms of the ball.
\begin{tw}
    Let $a,b \in \mathbb{B}^n$. The geodesic-preserving maps $f\in GP(\mathfrak{G}^{\mathbb{B}^n}_a,\mathfrak{G}^{\mathbb{B}^n}_b)$ have the form $f=M_b \circ g \circ M_a^{-1}$, where $g$ is as in Theorem \ref{2.1}.
\end{tw}
\begin{proof}
   The proof is analogous to that of Theorem 7.
\end{proof}
 We briefly recall elements from Lempert theory needed for studying the case where the intersection points of the families of geodesics lie on the boundary of the unit ball $\partial\mathbb{B}^n$. 

 A map $f: \mathbb{D} \longrightarrow D$ is called a complex geodesic if there exists a holomorphic map $F : D \longrightarrow \mathbb{D}$ such that $F \circ f \in \mathrm{Aut}(\mathbb{D})$; such an $F$ is called a left inverse of $f$. We call a domain $D$ strongly linearly convex if $D$ has a $\mathcal{C}^2 $-smooth boundary and there exists a defining function $r$ of $D$ such that $$\sum_{j,k=1}^{n}\frac{\partial^2 r}{\partial z_j\partial\bar z_k}(a)X_j\bar X_k>\left|\sum_{j,k=1}^{n}\frac{\partial^2 r}{\partial z_j\partial z_k}(a)X_jX_k\right|,\quad a\in\partial D,\; X\in T_D^{\mathbb C}(a) \setminus \{0\},$$ where $T_D^{\mathbb C}(a) = \{X \in \mathbb{C}^n : \sum_{j=1}^n \frac{\partial r}{\partial z_j}(a)X_j=0 \}$. In a strongly linearly convex domain $D$, there exists a biholomorphism $\Phi :D \longrightarrow G$ onto a smooth domain $G \subset \mathbb{C}^n$, $\Phi(f(\xi)) = (\xi, 0, \dots, 0)$ $\xi \in \bar{\mathbb{D}}$ and $G \subset \mathbb{D} \times \mathbb{C}^{n-1}$. Throughout this section we use the fact that in strongly linearly convex domains complex geodesics are uniquely determined (up to an automorphism of $\mathbb{D}$) by their images. 
 
 We refer the reader to \cite{KZ} and \cite{JP} for background on complex geodesics and Lempert domains.

We begin by observing that in strongly linearly convex domains, geodesic preservation restricted to complex geodesics induces complex geodesic preservation. In particular, the unit ball is strongly linearly convex and this lemma will give us some of the tools required to study the boundary point case.

\begin{lem}
     Let $D$ and $G$ be bounded strongly linearly convex domains, and $\mathfrak{G}_1^D$ and $\mathfrak{G}_2^G$ be two families of geodesics in $D$ and $G$ respectively. If $f$ is a geodesic-preseving function between $\mathfrak{G}_1^D$ and $\mathfrak{G}_2^G$, then $f$ maps complex geodesics into complex geodesics.
\end{lem}

\begin{proof}
Let $f$ be a nonconstant mapping. Let $\gamma_1 \in \mathfrak{G}_1^D$ and $\gamma_2 \in \mathfrak{G}_2^G$ be such that $f(\gamma_1)\subset \gamma_2$.  It is well known that in strongly linearly convex domains Kobayashi geodesics lie on complex geodesics and complex geodesics are unique \cite{KW}.

Let $\phi_1$ and $\phi_2$ be complex geodesics such that $\gamma_1\subset \phi_1(\mathbb{D})$ and $\gamma_2\subset \phi_2(\mathbb{D})$. Since $\phi_2(\mathbb{D})$ is an analytic subset of $D$, there exist holomorphic functions $h_1,\ldots,h_m\in\mathcal{O}(D)$ such that $\phi_2(\mathbb{D})=\bigcap_{j=1}^m h_j^{-1}(\{0\})$. For each $j$, the map $h_j\circ f\circ \phi_1:\mathbb{D}\to\mathbb{C}$ is holomorphic. Since $f(\gamma_1)\subset \gamma_2\subset \phi_2(\mathbb{D})$, we have $h_j\circ f\circ \phi_1=0$ on $\phi_1^{-1}(\gamma_1)$. As $\phi_1^{-1}(\gamma_1)$ contains a real interval, it has an accumulation point in $\mathbb{D}$. Therefore, by the identity principle, $h_j\circ f\circ \phi_1\equiv 0$ on $\mathbb{D}$ for every $j$. Consequently, $f(\phi_1(\mathbb{D}))\subset \bigcap_{j=1}^m h_j^{-1}(\{0\})=\phi_2(\mathbb{D})$. Thus $f$ maps complex geodesics into complex geodesics.

\end{proof}

The next theorem allows us to reduce the problem from real geodesics to complex geodesics, whose images under the Cayley transform have a particularly simple form.

 Any complex geodesics $\phi$ in the unit ball with $e_1 = (1,0, \dots,0) \in \overline{\phi(\mathbb{D})}$ can be defined as the intersection of the unit ball with a complex line passing through $e_1$. Concretely, points on such a geodesic have the form $$(z_1,\alpha_1(z_1-1), \dots , \alpha_{n-1}(z_1-1)),$$ where $\alpha_i \in \mathbb{C}$. This geometrical aspect implies that the family of geodesic rays originating from $e_1$ lying on a complex geodesic creates a similar structure to a family of geodesics intersecting on the boundary of the disc. Moreover, every complex geodesic in the ball can be obtained as the image of the horizontal unit disc $\mathbb{D}\times \{0\}^{n-1}$ by an automorphism of the ball. For an exact formula of these automorphisms we refer the reader to \cite{Pie}.

The next goal is to study geodesic-preserving maps using the Cayley transform of the unit ball centered at $e_1$. This transform, 
$$C_{e_1}(z,z')=\left( i\frac{1+z_1}{1-z_1},  i\frac{z'}{1-z_1}\right),$$
is a biholomorphic map of the unit ball $\mathbb{B}^n$ onto a Siegel domain $\mathcal{S}_n = \{ (z_1,z') \in \mathbb{C} \times \mathbb{C}^{n-1} : \Im(z_1) > |z'|^2\} \subset \mathbb{C}^n$. 
Its inverse is $$C^{-1}_{e_1}(z,z')=\left(\frac{2z_1}{i+z_1}-1, \frac{2z'}{i+z_1}\right).$$

The Cayley transform sends complex geodesics in the ball to certain half-planes in the Siegel domain $\mathcal{S}_n$. For example, applying it to the geodesic for $|z_1|<1$
$$\phi (z_1) = (z_1, \alpha_1 ( z_1-1))$$ yields 
$$C_{e_1}(\phi(z_1))=\left(i\frac{1+z_1}{1-z_1}, i\alpha_1\right),$$
which traces out a horizontal half-complex line. 

This also implies that $\Im(i\frac{1+z_1}{1-z_1}) >0$. Therefore real geodesics become vertical real half-lines of the form $$(a+it, i\alpha),\ \  t \geq |\alpha|^2.$$
We refer to these images as Siegel geodesics.

Rephrased in this framework, the problem becomes: find holomorphic maps of the Siegel domain that send Siegel geodesics to Siegel geodesics.

This description immediately restricts the form of geodesic-preserving maps on the Siegel domain. Vertical half-lines must be mapped to vertical half-lines, so the first coordinate may only undergo affine changes in the vertical direction, while the remaining coordinates cannot depend on the first variable. These observations lead to the following characterization. 

\begin{tw}
    Geodesic-preserving maps for a family of geodesics intersecting at a point $a \in \partial\mathbb{B}^n$ preserve complex geodesics intersecting at that point.
\end{tw}

\begin{proof}
   We may assume that the boundary point is $e_1=(1,0,\ldots,0)$. Let $G^{B^n}_{e_1}$ denote the family of geodesics passing through $e_1$, and let $f$ be a geodesic-preserving function for this family. 
   
   Observe that every real geodesic in the unit ball is contained in a complex geodesic. Indeed, real geodesics in $B^n$ are intersections of real geodesics in complex lines with the unit ball, whereas intersections of complex lines with the ball are precisely the complex geodesics. 
   
   Let $\phi:\mathbb D\to B^n$ be a complex geodesic passing through $e_1$. Since complex geodesics through $e_1$ are intersections of the ball with complex lines passing through $e_1$, every real geodesic contained in $\phi(\mathbb D)$ belongs to the family $G^{B^n}_{e_1}$. Since $f$ preserves $G^{B^n}_{e_1}$, the restriction of $f$ to $\phi(\mathbb D)$ maps real geodesics into real geodesics. Therefore, by Lemma 16, $f(\phi(\mathbb D))$ is contained in a complex geodesic. 
   
   Thus $f$ preserves complex geodesics intersecting at $e_1$. Conjugating with automorphisms of the ball gives the result for arbitrary $a\in\partial B^n$.
    
\end{proof}
   \begin{tw}\label{ball}
       Let $f:\mathbb B^n\to\mathbb B^n$ be a nonconstant geodesic-preserving map for the family of geodesics passing through $a \in \partial\mathbb{B}^n$. Let $C: \mathbb{B}^n \to \mathcal{S}_n$ denote the Cayley transform. Then $$f = C^{-1}\circ F\circ C,$$
       where $F:\mathcal{S}_n \to \mathcal{S}_n$ has the form 
       $$F(z_1,z')= (mz_1+b(z'), \tilde{F}(z')),$$
       for $(z_1,z') \in \mathcal{S}_n$, where $m >0$, $b \in \mathcal{O}(\mathbb{C}^{n-1},\mathbb{C})$ and $\tilde{F} \in \mathcal{O}(\mathbb{C}^{n-1})$ satisfies $$|\tilde{F}(z')|^2 \leq m | z'|^2+\Im (b(z')).$$ 
   \end{tw}
   \begin{proof}
       We may assume $a = e_1$. 

By Theorem 17, $f$ preserves complex geodesics intersecting at $e_1$. Under the Cayley transform,
complex geodesics through $e_1$ become horizontal complex half-planes of the form
$$\{(z_1,z')\in\mathcal S_n:z'=a\},$$
where $a\in\mathbb C^{n-1}$ is fixed. Therefore $F$ maps each such half-plane
into another half-plane of the same form. Hence the second component of $F = (F_1, F_2)$ is constant on every such half-plane, so
$$F_2(z_1,z')=\tilde{F}(z')$$
for some holomorphic map $\tilde{F}:\mathbb C^{n-1}\to\mathbb C^{n-1}$.

Now fix $z'\in\mathbb C^{n-1}$. The restriction of $F$ to the half-plane
$$\{(z_1,z'): \Im(z_1)>|z'|^2\}$$
preserves the Siegel geodesics ending at infinity. By the one-dimensional boundary classification, namely the analogue of Theorem 8 after the Cayley transform, the  first component must be affine in $z_1$ with positive real coefficient. Thus
$$F_1(z_1,z')=m(z')z_1+b(z'),$$
where $m(z')>0$ is real-valued.

Since $F_1$ is holomorphic in $z'$ and $m(z')$ is real-valued, $m(z')$ must be constant. Hence

$$F_1(z_1,z')=mz_1+b(z'), \qquad m>0.$$

It remains only to impose the condition that $F$ maps $\mathcal S_n$ into
itself. Since

$$F(z_1,z')=(mz_1+b(z'),\tilde{F}(z')),$$
we need

$$|\tilde{F}(z')|^2 < \Im(mz_1+b(z'))$$
whenever

$$ |z'|^2<\Im(z_1). $$

Equivalently,

$$|\tilde{F}(z')|^2 < m\Im(z_1)+\Im b(z').$$
Taking infimum over all admissible $\Im(z_1)$, we obtain

$$ |\tilde{F}(z')|^2 \leq m|z'|^2+\Im b(z').$$

Conversely, if $F$ has the form

$$ F(z_1,z')=(mz_1+b(z'),\tilde{F}(z'))$$
with $m>0$ and then

$$ |\tilde{F}(z')|^2 \leq m|z'|^2+\Im b(z'),$$
for every $(z_1,z')\in\mathcal S_n$,

$$|\tilde{F}(z')|^2 \leq m|z'|^2+\Im b(z') < m\Im(z_1)+\Im b(z')=\Im(F_1(z_1,z')).$$

Thus $F(\mathcal S_n)\subset\mathcal S_n$. Moreover, it maps horizontal half-planes into horizontal half-planes and, by the one-dimensional boundary classification on each such half-plane, maps vertical Siegel geodesics into vertical Siegel geodesics. Therefore $F$ is geodesic-preserving. Conjugating back by $C$ gives the result for $f$.

   \end{proof}

As in the interior case, these results can be extended to any family of geodesics passing through a point on the unit sphere using automorphisms of the closed unit ball $\overline{\mathbb{B}^n}$ sending the intersection point to $e_1.$
\begin{tw}
    Let $\zeta,\ \xi \in \partial\mathbb{B}^n$. Geodesic-preserving maps $f:\mathbb{B}^n \rightarrow \mathbb{B}^n$ between the families $\mathfrak{G}^{\mathbb{B}^n}_\zeta$ and $\mathfrak{G}^{\mathbb{B}^n}_\xi$ are of the form $$f=\Phi_{\xi}^{-1} \circ g \circ \Phi_\zeta,$$ where $g$ is a function as in Theorem \ref{ball}, $\Phi_\zeta,\Phi_\xi \in \aut(\overline{\mathbb{B}^n})$ map $\zeta$ and $\xi$ respectively to $e_1$.
\end{tw}

\section{Further research}
 
This work provides a detailed description of geodesic-preserving holomorphic functions on the disc and the unit ball for specific families of geodesics, but several natural cases remain open. Notably, we have not addressed the situation in which one family intersects in the unit sphere while the other intersects at the origin. One may also ask for a description of geodesic-preserving maps for arbitrary complete families of geodesics in these domains.

Additionally, extensions to more general domains in $\mathbb{C}$ and $\mathbb{C}^n$ may be possible. Because the definition of a geodesic-preserving function is quite broad, further exploration in diverse geometric or metric settings appears promising.

\section*{Acknowledgement}

This study was partially funded by the National Science Center, Poland, SHENG III, research project 2023/48/Q/ST1/00048.

The author would like to thank Łukasz Kosiński for his valuable help and the anonymous referee for the constructive suggestions and remarks which significantly improved the quality and readability of the paper.
\section*{Disclosure statement}

 The author reports there are no competing interests to declare.

Marcin Tombiński,

 Doctoral School of Exact and Natural Sciences

 Institute of Mathematics,
 
 Faculty of Mathematics and Computer Science,
 
 Jagiellonian University
 
 Lojasiewicza 6
 
 PL30348, Cracow, Poland
 
 marcin.tombinski@doctoral.uj.edu.pl
\end{document}